\documentclass[11pt]{amsart}

\usepackage{geometry}
\usepackage{mathtools}
\usepackage{amsmath,amssymb,color, array, verbatim}
\usepackage{enumerate}
\usepackage{verbatim}

\newtheorem{thm}{Theorem}

\newtheorem{cor}[thm]{Corollary}
\newtheorem{lem}[thm]{Lemma}

\newtheorem{prop}[thm]{Proposition}

\theoremstyle{definition}

\newtheorem{rem}[thm]{Remark}

\numberwithin{equation}{section}

\newcommand{\Tr}{{\mathrm{Tr}}}
\newcommand{\diag}{\mathrm{diag}}

\newcommand{\F}{\mathbb{F}}

\begin{document}

\title{Companion matrices as sums of $p$-potent and nilpotent matrices}

\author[A. Pojar]{Andrada Pojar}
\address{Technical University of Cluj-Napoca, Department of Mathematics, Str. Memorandumului 28, 400114, Cluj-Napoca, Romania}
\email{andrada.pojar@math.utcluj.ro}

\begin{abstract}
We prove that, over a field $\F$ of odd characteristic $p,$ a companion matrix $C$  is the sum of $E$ and $N,$ with $E$ $p$-potent (i.e. $E^p=E,$) and $N$ nilpotent, if and only if the trace of $C$ is an integer multiple of unity of $\F.$
\end{abstract}



\maketitle

\section{Introduction}

Let $\F$ be a field with positive odd characteristic $p.$
As usual, the letter $\F_p=\mathbb{Z}_p$ will stand for the prime field of $p$ elements having characteristic $p$, for any positive integer $n$, the notation $\mathbb{M}_n(\F)$ will denote the full matrix ring of $n\times n$ matrices over $\F.$

A square matrix $E$ over $\F$ is $p$-potent if $E^p=E.$ An idempotent is a $2$-potent. A square matrix $A$ over $\F$ is nilpotent if there exists $k$ a natural number such that $N^k=0.$ A square matrix $A$ over $\F$ is nil-clean if there exist an idempotent $E$ and a nilpotent $N$ such that $A=E+N.$ We consider decompositions of $A$ such that there exist a $p$-potent $E$ and a nilpotent $N,$ with $A=E+N.$

Nil-clean decompositions are related to clean ones, introduced by Nicholson in \cite{N}, when investigating exchange rings and were first studied by Diesl in \cite{Diesl}. An important result appeared in \cite{BCDM} and \cite{KLZ} about them is: every $n\times n$ matrix over a division ring $D$ is nil-clean if and only if $D=\F_2.$

An extension of nil-clean decompositions to finite fields of odd cardinality $q,$ was done in \cite{AM}: every matrix over such a field is the sum of a $q$-potent $E=E^q$ and a nilpotent. In \cite{B} there is even more -- the nilpotent $N$ involved in such a decomposition can be with nilpotence index at most $3$ (i.e. $N^3=0).$ Decompositions as a sum of a potent and square-zero matrix were considered in \cite{DGGL}.

In \cite{BM} it has been proved that if $\F$ is a field of positive characteristic, $p,$ then a companion matrix $A\in \mathbb{M}_n(\F)$ is nil-clean if and only if $A$ is nilpotent or unipotent or the trace of $A$ is of the form $t\cdot 1,$ with $t\in \{1,2,\dots,p\}$ such that $n>t$. In \cite{C} there is a characterization of $n\times n$ companion matrices over fields of positive odd characteristic $p$ that are sums of $m$ idempotents, $m\geq 2$,  and a nilpotent, in terms of dimension, and trace of such a matrix, and of $p.$ We prove that if the characteristic of the field is odd then a companion matrix $C\in \mathbb{M}_n(\F)$ is sum of a $p$-potent and a nilpotent if and only if the trace of $A$ is of the form $t\cdot 1,$ with $t$ an integer number, if and only if $C$ is sum of $m$ $p$-potents and a nilpotent.

\section{Modified companion matrices as a tool}
We provide our desired decomposition for companion matrices because each squared matrix over a field is similar to a direct sum of companion matrices. Now since a direct sum of matrices which can be decomposed as a sum of a $p$-potent and a nilpotent is a matrix with the same property and because of the fact that a matrix similar to a a $p$-potent is a $p$-potent and a matrix similar to a nilpotent is a nilpotent, our wanted decomposition for our initial matrix can be obtained.

Letting $q$ be a monic polynomial (i.e., its leading coefficient is $1$) over $\F,$ with $q=X^n+c_{n-1}X^{n-1}+\ldots+c_1X+c_0$, we explicitly indicate the companion matrix associated to $q$ as the $n\times n$ matrix
$$C=C_{c_0,c_1,\ldots, c_{n-1}}=\left(\begin{array}{ccccc}
0 & 0  &\ldots & 0 & -c_0 \\
1 & 0  &\ldots & 0 & -c_1\\
\vdots & \vdots  &\cdots  & \vdots & \vdots \\
0 & 0 &\ldots & 1 & -c_{n-1}
 \end{array}\right)$$

\noindent To avoid some inaccuracies with the exact meaning, we also denote $C$ by $C_q$ using the subscript $q$ which may vary in each of the different cases.

\medskip

Let $k$ a nonnegative integer, $k\leq n,$ and $\F$ a field. Let $a$ be a nonzero integer multiple of the unity of $\F$. A type of modified companion matrix $\widetilde{C}_{q,k,a}$ associated to q, namely
$\widetilde{C}_{q,k,a}=C_q+\diag (\underbrace{a,\dots,a }_{k \textrm{-times}} ,0,\dots,0),$ over $\F,$  has been defined in \cite{C} for $a=1$.

Let $k$ be nonnegative integer, $k\leq n,$ and $\F$ a field. Let $a_1,a_2,\dots,a_k$ be integer multiples of the unity of $\F$. We will work with a generalization of modified companion matrix  associated to q, namely
$C_q+\diag (a_1,a_2,\dots,a_k,0,\dots,0),$ over $\F.$

The next Lemma is a generalization of a Lemma in \cite{Pojar}.

\begin{lem}\label{use1}
 Let $k$ and $n$ be nonzero natural numbers, $n\geq k,$ and $\F$ be a field. Let $a_1,a_2,\dots,a_k$ be integer multiples of the unity. Every companion matrix $C_q\in M_n(\mathbb{F})$ is similar to the modified companion matrix  $C_{q'}+\diag (a_1,a_2,\dots,a_k,0,\dots,0)$ for some polynomial $q'.$
\end{lem}
\begin{proof}
First we will prove the statement for $k\in \{1,2,\dots,n-1\}.$
 Let $V$ denote the $n-$dimensional vector space of columns over $\mathbb{F}$ and consider $C_q$ as an endomorphism $C_q:V\rightarrow V.$ Denoting by
$\{e_1,e_2,\ldots,e_n\}$ the standard basis of $V,$ $C_q$ maps each $e_i$ to $e_{i+1}$, for each $i\in\{1,2,\ldots,n-1\}.$

Now we define $\{f_1,f_2,\ldots,f_n\},$ $f_i\in V,$ $i\in\{1,2,\ldots,n-1\},$ inductively as it follows. First set $f_1=e_1$. Assuming that $2\leq i\leq n$ and that $f_{i-1}$ has been defined, set $f_i=C_q(f_{i-1})-a_{i-1}f_{i-1},$ if $i\in \{2,\ldots,k+1\}$  and $f_i=C_q(f_{i-1}),$ if $i\in \{k+2,\ldots,n\}$

We have $e_1=f_1,$ so $e_1\in \textrm{Lin}(\{f_1\})$ and $f_2=C_q(f_1)-a_1f_1=C_q(e_1)-a_1f_1=e_2-a_1f_1,$ so $e_2=a_1f_1+f_2$ and $e_2\in \textrm{Lin}(\{f_1,f_2\})$

It is easy to see that each $f_i$ is the sum of $e_i$ and a linear combination of $e_{i-1},e_{i-2},\ldots,e_2,e_1$. Hence $e_i$ is the difference of $f_i$ and a linear combination of $e_{i-1},e_{i-2},\ldots,$ $e_2,e_1$. Assuming $e_1,e_2,\ldots,e_{i-1}\in \textrm{Lin}(\{f_1,f_2,\ldots,f_n\}),$ we get $e_i$ is a linear
combination of $f_1,f_2,\ldots,f_n$. Therefore $\textrm{Lin}(\{e_1,e_2,\ldots,e_n\})=\textrm{Lin}(\{f_1,f_2,\ldots,f_n\})$ and thus $\{f_1,f_2,\ldots,f_n\}$ is a basis of $V.$

Moreover by the definition we have:
$$C_q(f_1)=a_1f_1+f_2,$$
$$C_q(f_2)=a_2f_2+f_3,$$
$$ \vdots $$
$$C_q(f_k)=a_kf_k+f_{k+1},$$
$$C_q(f_{k+1})=f_{k+2},$$
$$ \vdots $$
$$C_q(f_{n-1})=f_n.$$
Let $M$ be the matrix the endomorphism $C_q$ corresponds to, with respect to the basis $B=\{f_1,f_2,\ldots,f_n\}$. Therefore
 $$M=[[C_q(f_1)]_B,\ldots,[C_q(f_{k})]_B,[C_q(f_{k+1})]_B,\dots,[C_q(f_n)]_B].$$
 Hence
 $$M=[[a_1f_1+f_2]_B,\ldots,[a_kf_k+f_{k+1}]_B,[f_{k+2}]_B,\ldots,[f_n]_B,[C_q(f_n)]_B].$$
 It follows that $M=\diag (a_1,\dots,a_k ,0,\dots,0)+C_{q'}$ for some monic polynomial $q'$ of degree $n.$
\\ So $C_q=P(\diag (a_1,\dots,a_k ,0,\dots,0)+C_{q'})P^{-1},$ where $P$ is the transition matrix mapping each $e_i$ to $f_i.$

 As next step we will solve  the case $k=n,$ that is we will prove that $C$ is similar to $\diag(a_1,a_2,\dots,a_n)+C',$ where $C'$ is a companion matrix. First, we know by the particular case of this Lemma (which is in \cite{Pojar}) that there exists a companion matrix such that $C\sim a_nI_n+C_1.$ Then by the case $k<n,$ we know that for $C_1$ there exists a companion matrix $C'$ such that $C_1\sim \diag(a_1-a_n,a_2-a_n,\dots,a_{n-1}-a_n,0)+C'.$ Hence $C$ is similar to $\diag(a_1,a_2,\dots,a_n)+C'.$

\end{proof}

\section{Companion matrices as sums of $p$-potent and nilpotent matrices}
The following Proposition has been proved for $\F_p$ in \cite{Pojar}. The same proof is valid for any field with with positive odd characteristic $p.$

\begin{prop}\label{PrescrDecomp}
Let $\F$ be a field with positive odd characteristic $p.$
Let $n\geq 2$ and $k$ be positive integers such that $k<n.$ Let $a\in \{1,2,\dots,p-2\}.$ Fix constants $d_0, d_1,\dots,d_{n-1}\in \F$ and denote $D=C_{d_0,d_1,\ldots,d_{n-1}}+\diag (\underbrace{a,\dots,a }_{k \textrm{-times}}  ,0,\dots,0).$ For every polynomial $g\in \F[X]$ of degree at most $n-2$ there exist two matrices $E,$ $M$ in $\mathbb{M}_n(\F)$ such that:
\begin{enumerate}
\item $D=E+M$
\item $E^p=E$, and
\item $\chi_M=X^n+(k\cdot 1+d_{n-1})X^{n-1}+g.$
\end{enumerate}
\end{prop}

\begin{cor}\label{corPrescrDecomp}
Let $\F$ be a field with positive odd characteristic $p.$ Let $n\geq 2$ be a positive integer. Fix constants $c_0, c_1,\dots,c_{n-1}\in \F.$
Let $C_{c_0,c_1,\ldots,c_{n-1}}$ be a companion matrix and if there exist $k\in \{1,2,\ldots,n-1\}$ and $a\in \{1,2,\ldots,p-2\}$ such that $-c_{n-1}=k(a+1)$ then  $C$ is sum of a $p$-potent and a nilpotent.
\end{cor}
\begin{proof}
Let $C_{c_0,c_1,\ldots,c_{n-1}}$ be a companion matrix. By Lemma \ref{use1} we have that there exist  $d_0, d_1,\dots,d_{n-1}\in \F$, $k\in \{1,2,\ldots,n-1\},$ and $a\in \{1,2,\ldots,p-2\}$ such that $C_{c_0,c_1,\ldots,c_{n-1}}$ is similar to $D=C_{d_0,d_1,\ldots,d_{n-1}}+\diag (\underbrace{a,\dots,a }_{k \textrm{-times}}  ,0,\dots,0)$. The traces of two similar matrices are equal and so $-c_{n-1}=-d_{n-1}+k\cdot a$. But $-c_{n-1}=k(a+1)$, and so $-d_{n-1}=k\cdot 1.$ Now taking $g=0$ in Proposition \ref{PrescrDecomp} we have that there exist a $p$-potent matrix $E$ and a nilpotent matrix $N$, such that $D=E+N$. Therefore since $C_{c_0,c_1,\ldots,c_{n-1}}$ is similar to $D$ and because a matrix similar to a $p$-potent is a $p$ potent and a matrix similar to a nilpotent is a nilpotent, we have a decomposition as a sum of a $p$-potent and a nilpotent for $C_{c_0,c_1,\ldots,c_{n-1}}.$
\end{proof}

\begin{lem}\label{traceless}
Let $\F$ be a field with positive odd characteristic $p.$ Let $n\geq 1$ be a positive integer. Fix constants $c_0, c_1,\dots,c_{n-2},0\in \F.$
Let $C=C_{c_0,c_1,\ldots,c_{n-2},0}$ be a companion matrix. Then $C$ is the sum of a $p$-potent matrix and a nilpotent matrix.
\end{lem}
\begin{proof}

\begin{itemize}
\item $n=1,$ C=0=0+0 is a sum of a $p$-potent and a nilpotent.
\item $n=2,$ let $$C=\left(\begin{array}{cc}
0 & -c_0 \\
1 & 0
\end{array}\right)=\left(\begin{array}{cc}
0 & 1 \\
1 & 0
\end{array}\right)+\left(\begin{array}{cc}
0 & -c_0-1 \\
0 & 0
\end{array}\right)$$ is the sum of a matrix with square the identity and a nilpotent, hence this is the sum of a $p$-potent and a nilpotent.
\item $n=2k+1,$ with $k\geq 1$ integer; let $a\neq 0$ be an integer multiple of unity of $\F.$ By Lemma \ref{use1} there exists a companion matrix $C'_{c'_0,c'_1,\ldots,c'_{n-1}}$ over $\F$ such that $C\sim D$, where $D=\diag (\underbrace{a,\dots,a }_{k \textrm{-times}} ,\underbrace{-a,\dots,-a }_{k \textrm{-times}} ,0)+C'_{c'_0,c'_1,\ldots,c'_{n-1}}.$
    For $D$ we have the following decomposition
    $$D=\left(\begin{array}{ccc}
aI_k & 0 & U \\
0 & -aI_k & V\\
0 & 0 & 0
\end{array}\right)+C"_{0,0,\dots,0},$$ where

$U=\begin{bmatrix}
   -c'_0\\
   -c'_1\\
   \vdots\\
   -c'_{k-1}
  \end{bmatrix}$
 and
 $V=\begin{bmatrix}
   -c'_k\\
   -c'_{k+1}\\
   \vdots\\
   -c'_{2k-1}
  \end{bmatrix}$
The above decomposition for $D$ is the sum of a $p$-potent and a nilpotent. Now since a matrix similar to a $p$-potent is a $p$-potent and a matrix similar to a nilpotent is a nilpotent, we find that there exists a decomposition as a sum of a $p$-potent and a nilpotent for $C,$ which is similar to $D$.
\item $n=2k+2,$ with $k\geq 1$ integer; let $a\neq 0$ be an integer multiple of unity of $\F.$ By Lemma \ref{use1} there exists a companion matrix $C'_{c'_0,c'_1,\ldots,c'_{n-1}}$ over $\F$ such that $C\sim D$, where $D=\diag (\underbrace{a,\dots,a }_{k \textrm{-times}} ,\underbrace{-a,\dots,-a }_{k-1 \textrm{-times}} ,-a-1,1,0)+C'_{c'_0,c'_1,\ldots,c'_{n-1}}.$
    For $D$ we have the following decomposition
    $$D=\left(\begin{array}{cc}
P & U \\
0 & 0
\end{array}\right)+C"_{0,0,\dots,0},$$ where

$$P=\diag (\underbrace{a,\dots,a }_{k \textrm{-times}} ,\underbrace{-a,\dots,-a }_{k-1 \textrm{-times}} ,-a-1,1),$$ and
$U=\begin{bmatrix}
    -c'_0\\
    -c'_1\\
   \vdots\\
    -c'_{2k}
  \end{bmatrix}$

The above decomposition for $D$ is the sum of a $p$-potent and a nilpotent. Now since a matrix similar to a $p$-potent is a $p$-potent and a matrix similar to a nilpotent is a nilpotent, we find that there exists a decomposition as a sum of a $p$-potent and a nilpotent for $C,$ which is similar to $D$.
\end{itemize}
\end{proof}

\begin{thm}\label{mainThm}
Let $\F$ be a field with positive odd characteristic $p.$ Let $n\geq 1$ be a positive integer.
Let $C$ be a companion matrix. Then $C$ is sum of a $p$-potent and a nilpotent if and only if the trace of $C$ is an integer multiple of unity of $\F.$
\end{thm}
\begin{proof}
Let $C=C_{c_0,c_1,\ldots,c_{n-1}}$ be a companion matrix.

Let n=1 and $-c_0\in \F$. If $-c_0$ is an integer multiple of unity then there exist the $p$-potent $-c_0$ and the nilpotent $0\in F_p$ such that $-c_0=-c_0+0$. Conversely if $-c_0=e+n$, $e,n\in \F$, $e^p=e$ and $n$ is a nilpotent then $n=0$ and $-c_0=e.$ Therefore $-c_0$ verifies the equation $X^p=X$ which has the $p$ solutions all $p$ integer multiples of unity of $\F.$ Therefore $-c_0$ is an integer multiple of unity of $\F.$

Now let $n\geq 2.$

 Let $E$ be $p$-potent and $N$ be nilpotent such that $C=E+N,$ then since the trace of a nilpotent is $0$ then $\mbox{trace}(C)=\mbox{trace}(E),$ that is $-c_{n-1}$ is the sum of the eigenvalues of $E$. Let $m_E$ be the minimal polynomial of $E$. Then $m_E$ divides $X^p-X,$ and therefore if $\lambda$ is an eigenvalue of $E,$ then $\lambda^p=\lambda.$ Hence $\lambda$ is an integer multiple of unity and the sum of the eigenvalues of $E$ is an integer multiple of unity. Therefore $-c_{n-1}$ is an integer multiple of unity of $\F.$

Conversely let $-c_{n-1}$ be an integer multiple of unity of $\F.$

Let $-c_{n-1}=0$. Then by Lemma \ref{traceless} we have that $C$ is sum of a $p$-potent and a nilpotent.

Let $-c_{n-1}\neq 0,$ then we choose $k\in \{1,2,\ldots,n-1\},$ $k\neq 0\cdot 1,$ and $k\neq -c_{n-1}$, and $a=k^{-1}(-c_{n-1})-1\in \{1,2,\ldots,p-2\},$  then $-c_{n-1}=k(a+1)$. Hence by Corollary \ref{corPrescrDecomp} we have that $C$ is sum of a $p$-potent and a nilpotent.

In conclusion $C$ is sum of a $p$-potent matrix and a nilpotent matrix.
\end{proof}

\begin{thm}\label{mppot}
Let $\F$ be a field with positive odd characteristic $p.$ Let $n\geq 1,$ and $m\geq 2$ be positive integers. Let $C$ be a companion matrix. Then $C$ is sum of $m$ $p$-potents and a nilpotent if and only if the trace of $C$ is an integer multiple of unity of $\F.$
\end{thm}
\begin{proof}
Let $C$ be a companion matrix which has a decomposition $C=E_1+\dots+E_m+N,$ with all $E_i$ being $p$-potents, and $N$ being nilpotent.
It follows that $$\Tr(C)=\Tr(E_1)+\dots+\Tr(E_m)$$
Every eigenvalue of every $E_i,$ verifies the equation, $X^p=X,$ therefore is an integer multiple of unity of $\F$, and so are the traces of all $E_i,$ and $\Tr(C).$

Assume now that the trace of the companion matrix $C$ is an integer multiple of unity of $\F.$ Fix integer multiples of unity of $\F$,$a_1,a_2,\dots,a_k.$ By Lemma \ref{use1} we have that there exists a companion $C'$ such that
$$C\sim C'+\diag (a_1,a_2,\dots,a_k,0,\dots,0).$$
Then the trace of $C'$ is an integer multiple of unity of $\F$, and therefore, by Theorem \ref{mainThm}, $C'$ is sum of a $p$-potent and of a nilpotent.

 Now since a matrix similar to a $p$-potent is $p$-potent, and a matrix similar to a nilpotent is nilpotent, we derive that $C$ is sum of two $p$-potents and a nilpotent. Using the lemma we used before and the induction hypothesis that a companion matrix with trace an integer multiple of unity is sum of $m\geq 2,$ $p$-potents and a nilpotent, we derive that $C$ is sum of $m+1$ $p$-potents and a nilpotent, thus proving $C$ is sum of $m,$ $p$-potents and a nilpotent, for every $m\geq 2.$
\end{proof}

By Theorems \ref{mainThm} and \ref{mppot} we derive the following result:

\begin{cor}
Let $\F$ be a field with positive odd characteristic $p.$ Let $n\geq 1,$ and $m\geq 2$ be positive integers.
Let $C$ be a companion matrix. Then $C$ is sum of $m$ $p$-potents and a nilpotent if and only if $C$ is sum of a $p$-potent and a nilpotent.
\end{cor}

\begin{rem}
By Remark $9$ from \cite{BM} it is known that if all companion matrices which appear in the Frobenius normal form of a matrix $A$ are sums of idempotent and nilpotent matrices, then $A$ is sum of an idempotent and a nilpotent, and it is not known if the converse is true. Comparatively, the situation in case of decompositions as matrices over a field of characteristic $p>2$ as sums of $p$-potent and nilpotent matrices is in the following way: it is not hard to see that if all companion matrices which appear in the Frobenius normal form of a matrix $A$ are sums of $p$-potent and nilpotent matrices, then $A$ is sum of a $p$-potent matrix and a nilpotent matrix. The converse is true in case $\F=\F_p,$ since all companion matrices over this field are sums of $p$-potent and nilpotent matrices. It would be nice to know if the converse is true for any other field.
\end{rem}

\end{document}